\documentclass[a4paper,11pt]{amsart}
\usepackage{latexsym}
\usepackage{amssymb,amsmath}
\usepackage{graphics}
\usepackage{url}
\newtheorem{theorem}{Theorem}
\newtheorem{corollary}[theorem]{Corollary}
\newtheorem{proposition}[theorem]{Proposition}
\newtheorem{lemma}[theorem]{Lemma}
\theoremstyle{definition}
\newtheorem{definition}[theorem]{Definition}
\newtheorem{remark}[theorem]{Remark}
\newcommand{\N}{\mathbb{N}}
\newcommand{\Z}{\mathbb{Z}}
\newcommand{\Q}{\mathbb{Q}}
\newcommand{\C}{\mathbb{C}}

\DeclareMathOperator{\Sym}{Sym}

\DeclareMathOperator{\Span}{Span}
\newcommand\blfootnote[1]{%
  \begingroup
  \renewcommand\thefootnote{}\footnote{#1}%
  \addtocounter{footnote}{-1}%
  \endgroup
}
\newcounter{thmlistcnt}
	{\setcounter{thmlistcnt}{0}%
	\begin{list}{\emph{(\roman{thmlistcnt})}}{%
		\usecounter{thmlistcnt}%
		\setlength{\topsep}{0pt}%
		\setlength{\leftmargin}{0pt}%
		\setlength{\itemsep}{0pt}%
		\setlength{\itemindent}{17pt}}%
	}%
	{\end{list}}%

\begin{document}
\title[The Action of Generalised Symmetric Groups]{The Action of Generalised Symmetric Groups on Symmetric and Exterior Powers of Their Natural Representations}
\date{\today}\author{Bill O'Donovan}
\begin{abstract}
We establish upper and lower bounds on the dimension of the space spanned by the symmetric powers of the natural character of generalised symmetric groups. We adapt the methods of Savitt and Stanley from [1] to obtain bounds  both over the complex numbers and in prime characteristic. 
\end{abstract}

\maketitle
\thispagestyle{empty}

\section{Introduction}

In \blfootnote{2010 \textit{Mathematics Subject Classification}. 20C30

\textit{Key words and phrases}. Symmetric power, wreath product, exterior power, generalised symmetric group, Brauer character.}[1], Savitt and Stanley proved that the space spanned by the symmetric powers of the standard $n$-dimensional complex representation of $S_n$ has dimension asymptotic to $\frac{n^2}{2}$. In this paper, we look at some extensions of this, as well as demonstrate the wide applicability of the methods employed in [1] for upper bounds. The methods used in [1] to obtain lower bounds are not so readily applicable, but some partial results can be found. The study of the symmetric powers of the natural representation of $S_n$ is motivated in part by the fact that if $\lambda$ is a partition and $r \geq \sum_{i=1}^{\lambda_1} \binom{\lambda_i'}{2}$, then $\Sym^rV$ contains the Specht module $S^\lambda$.  

In Section 2 of this paper, we consider the action of the generalised symmetric groups groups $C_k \wr S_n$, on the symmetric and exterior powers of their natural representations as $n \times n$ complex matrices. We prove that, to leading order, the space spanned by the characters of the symmetric powers of the natural representation of $C_k \wr S_n$ has dimension between $kn\log n $ and $\frac{kn^2}{2}$. If $\lambda = (\lambda_1, \ldots, \lambda_s)$ is a partition, then we define $f_{\lambda}(q) := \prod_{j=1}^s (1-q^{\lambda_j})^{-1}.$ The identity (where $\sigma \in S_n$ has cycle type $\lambda$) $$ \sum_{r} (\Sym^r \pi) (\sigma) q^r = p_\lambda(1,q,q^2, \ldots ) = f_\lambda(q),$$ plays an important part of the analysis in [1]; in Section 2 we obtain a generalisation of this for wreath products. Furthermore, we show that if $k \geq 2$, then the space spanned by the characters of the exterior powers of this representation has maximum possible dimension, namely $n+1$. In contrast, the space spanned by the exterior powers of the natural character of $S_n$ (the case when $k=1$) has dimension $n$.

In Section 3 of this paper, we examine the same issue explored in [1], but this time in positive characteristic, and prove that the space spanned by the symmetric powers of the natural representation of $S_n$ in characteristic $p$ is bounded (to leading order) by $\frac{p}{2p+2}n^2$. Although the argument for lower bounds given in [1] cannot be modified for this case, we use a different method to establish a lower bound of (again to leading order) $\frac{p-1}{p} n\log n$. Numerical data suggests that this lower bound can be improved upon and that the upper bound is closer to the truth. We conclude by putting both extensions together: we consider the symmetric powers of $C_k \wr S_n$ in characteristic $p$ and prove that the dimension of the space spanned by the symmetric powers is bounded above, to leading order, by $r_p(k)\frac{p}{2p+2}n^2$, where $r_p(k)$ is the $p'$-part of $k$, as defined in Definition 9.

\section{Generalised Symmetric Groups}
\subsection{Symmetric Powers}

We consider an extension of [1], namely to the generalised symmetric group $C_k \wr S_n$ for $k \in \N$; we think of this group as the subgroup of $GL_n(\C)$ consisting of matrices whose entries are $k^{\text{th}}$ roots of unity, and which have exactly one non-zero entry in each row and column. 

\begin{definition}Given $\sigma \in C_k \wr S_n$, we define $|\sigma|$ to be the element of $S_n$ corresponding to the permutation matrix we obtain by setting all of the non-zero entries of $\sigma$ equal to one. We also define a map $\lambda : C_k \wr S_n \rightarrow \text{Par}(n)$ by sending $\sigma \in C_k \wr S_n$ to the cycle type of $|\sigma|$. If $\lambda(\sigma) =(\lambda_1,\ldots,\lambda_s)$, let $t_i \in \C$ denote the product of the entries in $\sigma$ which correspond to cycle $\lambda_i$ of $|\sigma|$. Where $\lambda(\sigma)$ has a repeated part, say $\lambda_i$, to resolve ambiguity, we order the $\lambda_i$-cycles in $|\sigma|$ by the minimum value of the support of the cycle.
 \end{definition}
For example, if $k =2$ and \[\sigma = \left( \begin{array}{ccccc}
0 & -1 & 0 & 0 & 0 \\
-1 & 0 & 0 & 0 & 0 \\
0 & 0 & 1 & 0 & 0 \\
0 & 0 & 0 & 0 & -1 \\
0 & 0 & 0 & 1 & 0 \end{array} \right)\] then $|\sigma| = (12)(3)(45) \in S_5$, $\lambda(\sigma) = (2,2,1)$, $t_1 = 1$, $t_2 = -1$ and $t_2 = 1$.The conjugacy class of $ \sigma \in C_k \wr S_n$ is determined by $\lambda(\sigma) = (\lambda_1,\ldots,\lambda_s)$ and the $k$-tuple $(t_1,\ldots,t_s)$. Also, such pairs $(\lambda,t)$ where $\lambda = (\lambda_1,\ldots,\lambda_s) \in \text{Par}(n)$ and $t=(t_1,\ldots,t_s)$, where each $t_i$ is a $k^{\text{th}}$ root of unity, parametrise the conjugacy classes. For details see [4, Section 4.2].

Let $V$ be the natural representation of $C_k \wr S_n$ with basis $e_1,\ldots,e_n$ and let $\Sym^rV$ be its $r^{th}$ symmetric power with character $\chi_r$. Let $\sigma \in C_k \wr S_n$ and consider the action of $\sigma$ on $\Sym^rV$. Now, $\Sym^rV$ has a basis $\lbrace e^{c_1}_1 \ldots e^{c_n}_n : \sum_i c_i = r \rbrace$, and we want to know when a basis vector contributes to the trace of the matrix of $\sigma$ acting on $\Sym^rV$. Observe that, with this basis, $\sigma$ always maps a basis vector to a multiple of another basis vector; consequently, the only contribution to the trace comes from those basis vectors which are eigenvectors for the action of $\sigma$. 

Suppose, without loss of generality, that $|\sigma| = (1, 2, \ldots ,\lambda_1) \ldots (\lambda_1 + \cdots + \lambda_{s-1}+1, \ldots ,n)$, so $\lambda(\sigma) = (\lambda_1, \ldots ,\lambda_s) $. Then the vector $v = e^{c_1}_1 \ldots e^{c_n}_n$ is an eigenvector for $\sigma$ if and only if $c_1 = \ldots = c_{\lambda_1}$, $c_{\lambda_1+1} = \ldots = c_{\lambda_1+\lambda_2}$ and so on. Hence the eigenvectors for $\sigma$ are in a one-to-one correspondence with the set of non-negative integral solutions $(c_1, \ldots, c_s)$ to $\lambda_1c_1 + \cdots +\lambda_sc_s = r$.

Suppose that $v$ is a basis vector which is an eigenvector for the action of $\sigma$. Then we have $\sigma .v = \prod^{s}_{i = 1} t_i^{c_i} v $, and so $\chi_r (\sigma) = \sum \prod^{s}_{i = 1} t_i^{c_i}$, where the sum runs over all the basis vectors of $\Sym^rV$ which are eigenvectors for the action of $\sigma$. It follows that $$\chi_r (\sigma) = \sum_{(c_1, \ldots, c_s)} \prod^{s}_{i = 1} t_i^{c_i},$$where the sum is over all solutions to the equation $\lambda_1c_1 + \cdots +\lambda_sc_s = r$. Define the power series $$f_\sigma (x) = \frac{1}{(1-t_1x^{\lambda_1}) \ldots (1-t_sx^{\lambda_s})}.$$
The coefficient of $x^m$ in the power series $f_{\sigma}(x)$ is  $$\sum_{(c_1,\ldots,c_s)} \prod^{s}_{i = 1} t_i^{c_i} = \chi_m (\sigma),$$ where the sum is over all solutions to the equation $\lambda_1c_1 + \cdots +\lambda_sc_s = r$. Consequently, this power series is the generating function for the values of $\chi_r (\sigma)$. 

Let $D(n,k)$ be the dimension of the space spanned by the characters of the symmetric powers of the natural representation for $C_k \wr S_n$, i.e. $D(n,k) = \dim (\Span_{\C} \lbrace \chi_r : r \in \N \rbrace)$. Consider the matrix of the characters $\chi_r$ for $r \in \N$; $D(n,k)$ is equal to the row-rank of this matrix. Since row-rank and column-rank are the same, $D(n,k)$ equals the dimension of the space spanned by the columns of the table. By the above argument, the generating function for the entries of the column corresponding to the conjugacy class of $\sigma \in C_k \wr S_n$ is $f_{\sigma}(x)$. We have therefore proved:
   
\begin{proposition}$D(n,k) = \dim (\Span_{\C} \lbrace f_\sigma(x) : \sigma \in C_k \wr S_n \rbrace).$ \end{proposition}

\begin{proposition}We have $D(n,k) \leq \frac{kn^2}{2} +(\frac{k}{2} - 1)n + 1$. \end{proposition}
\begin{proof}
Let $\zeta = \exp(\frac{2\pi i}{k})$ and consider $\lbrace f_\sigma(x) : \sigma \in C_k \wr S_n \rbrace$; we claim that a common denominator for these power series is $$ D(x) := \prod_{\substack{
      0 \le i < k \\
      1 \le j \le n}} (1-\zeta^ix^j).$$
      
Suppose that $\theta \in \C$ and $x-\theta$ divides $\frac{1}{f_\sigma(x)}$ in $\C[x]$, for some $\sigma \in C_k \wr S_n$. Let $r$ be the minimal natural number such that, in $\C[x]$, $x-\theta$ divides $1-\zeta^ix^r$ for some $0 \leq i < k$. Then, in a factorisation over $\C$ into linear factors of a common denominator for the $f_\sigma(x)$, the factor $x-\theta$ appears exactly $\lfloor \frac{n}{r} \rfloor$ times by minimality of $r$. Since $\zeta^i\theta^r = 1$, it follows that for any $m \in \N$, $\zeta^{im}\theta^{rm} = 1$, whence $x-\theta$ divides $1-\zeta^{im}x^{rm}$ in $\C[x]$. This is a factor of $D(x)$ for $m=1,\ldots,\lfloor \frac{n}{r} \rfloor$, so $(x-\theta)^{\lfloor \frac{n}{r} \rfloor}$ divides $D(x)$ in $\C[x]$, and this establishes the claim.
 
The elements $\frac{1}{D(x)},\frac{x}{D(x)},\ldots,\frac{x^{\text{deg}D-n}}{D(x)}$ form a spanning set for the vector space $\Span_{\C} \lbrace f_\sigma(x) : \sigma \in C_k \wr S_n \rbrace$, giving $D(n,k) \leq \text{deg } D + 1 -n$. Note that $D(x)$ has degree $k \sum_{j=1}^n j = \frac{kn}{2}(n+1)$, and the result follows.
\end{proof}

Although the method used in [1, Section 3] to obtain a lower bound is not adaptable to this case (see also Remark 13), we can establish a lower bound by other methods. 

\begin{proposition}There exists a constant $c>0$ such that $D(n,k) \geq \max (\frac{n^2}{2} -cn^{\frac{3}{2}}, k((n-1)\log (n-1) -2(n-2))).$  \end{proposition}
\begin{proof}
Since $\lbrace f_\lambda(x) : \lambda \in \text{Par}(n) \rbrace $ is a subset of $\lbrace f_\sigma(x) : \sigma \in C_k \wr S_n \rbrace$, we get the lower bound $D(n,k) \geq \frac{n^2}{2} -cn^{\frac{3}{2}}$ for some constant $c$ by [1, Proposition 3.3]. Let $W_{n-1}$ denote the set of partitions of $n-1$ which have the form $(a^b,1^{n-1-ab})$, where $a,b \in \N_0$ and $a > 1$, and let $\zeta = \exp(\frac{2\pi i}{k})$. We claim that $Y := \lbrace \frac{1}{1-\zeta^rx} f_\lambda (x) : 0 \leq r < k, \lambda \in W_{n-1} \rbrace$ is a linearly independent set. We prove this by induction on the position of the partition $\lambda$ in the lexicographic ordering on Par($n-1$).

Suppose that for some constants $\alpha_{\lambda,r}$, we have $$\sum_{\lambda \in W_{n-1}} \sum_{r=0}^{k-1} \alpha_{\lambda,r} \frac{f_\lambda(x)}{1-\zeta^rx} = 0.$$ Let $\mu_1$ be the partition $(n-1) \in W_{n-1}$, then we can re-write this as $$-\sum_{r=0}^{k-1}  \frac{\alpha_{\mu_1,r}}{(1-\zeta^rx)(1-x^{n-1})} = \sum_{\lambda \in W_{n-1} \setminus \lbrace \mu_1 \rbrace} \sum_{r=0}^{k-1} \alpha_{\lambda,r} \frac{f_\lambda(x)}{1-\zeta^rx}.$$
The left-hand side has a singularity at $\exp(\frac{2\pi i}{{n-1}})$, whereas the right-hand side is holomorphic at this point. It follows that $\sum_{r=0}^{k-1} \frac{\alpha_{\mu_1,r}}{1-\zeta^rx} = 0.$ Since the function $\frac{1}{1-\zeta^rx}$ has a pole at $\zeta^{-r}$ only, we see that the power series $\frac{1}{1-\zeta^rx}$ for $ 0 \leq r < k$ are linearly independent. Hence $\alpha_{\mu_1,r} = 0$ for all $r$.

Let $\mu = (a^b,1^{n-ab-1})$ and suppose that $\alpha_{\lambda,r} = 0$ for every partition $\lambda \in W_{n-1}$ which is greater than $\mu$ in the lexicographic ordering. Then we can write

$$-\sum_{r=0}^{k-1}  \frac{\alpha_{\mu,r}}{(1-\zeta^rx)(1-x^{a})^b(1-x)^{n-ab-1}} = \sum_{\lambda < \mu} \sum_{r=0}^{k-1} \alpha_{\lambda,r} \frac{f_\lambda(x)}{1-\zeta^rx}.$$

The left-hand side has a pole of order $b$ at $\exp(\frac{2\pi i}{{a}})$. There are two cases. If the right-hand side is holomorphic at $\exp(\frac{2\pi i}{{a}})$, then we deduce that $\alpha_{\mu,r} = 0$ by the same argument given in the base case, and so we may suppose that the right-hand side has a pole at $\exp(\frac{2\pi i}{{a}})$. However, if the pole at $\exp(\frac{2\pi i}{{a}})$ on right-hand side were of order $b$ or greater, then there would be a partition $\lambda \in W_{n-1}$ with the part $a$ repeated at least $b$ times. By construction of the set $W_{n-1}$, this is impossible since $\lambda < \mu$.

Hence the right-hand side has a pole at $\exp(\frac{2\pi i}{{a}})$ of order strictly less than $b$. Consequently, we deduce that $\alpha_{\mu,r} = 0$. It follows by induction that all the $\alpha_{\lambda,r}$ are equal to zero. This proves the claim.

Therefore, we have $D(n,k) \geq k|W_{n-1}|$. It is clear that $$|W_m| = 1 + \sum_{t=2}^m \lfloor \frac{m}{t} \rfloor \geq 1 + \sum_{t=2}^m (\frac{m}{t} - 1).$$ Therefore, we have $|W_m| \geq 1 + mH_m - m - (m-1)$, where $H_m$ is the $m^{\text{th}}$ harmonic number, and so, using the simple bound $H_m \geq \log m$, it follows that $|W_m| \geq m\log m - 2(m-1)$. Consequently, we deduce that $D(n,k) \geq k((n-1)\log (n-1) -2(n-2))$, as required.
\end{proof}

\begin{remark}In the above proof we used that the power series $\frac{1}{1-\zeta^rx}$ for $r=1,\ldots,k$ are linearly independent, whence $D(1,k) = k$. In particular, the upper bound $D(n,k) \leq \frac{kn^2}{2} +(\frac{k}{2} - 1)n + 1$ is an equality when $n=1$. 
 \end{remark}

\subsection{Exterior Powers}

Let $\Lambda^rV$ be the $r^{\text{th}}$ exterior power of the natural representation of $C_k \wr S_n$ with basis $\lbrace e_{i_1} \wedge \ldots \wedge e_{i_k} : 1 \leq i_1 < \ldots < i_k \leq n \rbrace$ and character $\psi_r$.

Note that with this basis of $\Lambda^rV$, each $\sigma \in C_k \wr S_n$ again sends a basis vector to a multiple of a basis vector, so the only contribution to the trace of the matrix of $\sigma$ on $\Lambda^rV$ is from those basis vectors which are eigenvectors for the action of $\sigma$. Let a basis vector be $v = e_1^{c_1} \wedge e_2^{c_2} \wedge \ldots \wedge e_n^{c_n}$, where $c_i \in \lbrace 0,1 \rbrace$ and we interpret $e_i^0$ to mean that $e_i$ does not appear (for example, $e_1^1 \wedge e_2^0 \wedge e_3^1 = e_1 \wedge e_3$). Without loss of generality, suppose that $|\sigma| = (1,\ldots, \lambda_1)\ldots(\lambda_1+\cdots+\lambda_{s-1}+1,\ldots,n)$, so $\lambda(\sigma) = (\lambda_1, \ldots, \lambda_s)$.

Observe that, for every $j$ such that $c_j =1$, we must have $c_k = 1$ whenever $k$ is in the same cycle as $j$ in $|\sigma|$. Therefore, the eigenvectors correspond to solutions to the equation $\lambda_1c_1 + \cdots +\lambda_sc_s = r$ where each $c_i$ is either 0 or 1. Moreover, a $m$-cycle in $S_n$ with $c_i = 1$ on its support sends $v$ to $(-1)^{m-1}v$ since an $m$-cycle has sign $(-1)^{m-1}$.

Then $\sigma.v = \prod_{\text{$i$}}  (-1)^{\lambda_i -1} t_i,$ where the product is over $ 1\leq i \leq s$ such that $c_i = 1$ and $t_i$ is as in Definition 1. It follows that $$\psi_r(\sigma) = \sum_{(c_1, \ldots c_s)}\prod_{\text{$i $}}  (-1)^{\lambda_i -1} t_i,$$ where the sum is over solutions $(c_1, \ldots, c_s)$ to the equation $\lambda_1c_1 + \cdots +\lambda_sc_s = r$ and the product is again over $ 1\leq i \leq s$ such that $c_i = 1$.

Given $\sigma \in C_k \wr S_n$ with $\lambda(\sigma) = (\lambda_1,\ldots,\lambda_s)$ and associated $s$-tuple of $k^{\text{th}}$ roots of unity $(t_1,\ldots,t_s)$, we define $$g_\sigma(x) = \prod_{i = 1}^s (1-t_i(-1)^{\lambda_i -1}x^{\lambda_i}).$$

A similar argument to that used to prove Proposition 2 shows that $g_\sigma(x)$ is the generating function for the values of $\psi_r(\sigma)$. Let $E(n,k)$ be the dimension of the space spanned by the characters of the exterior powers of the natural representation for $C_k \wr S_n$. Then we have proved:

\begin{proposition}$E(n,k) = \dim (\Span_{\C} \lbrace g_\sigma(x) : \sigma \in C_k \wr S_n \rbrace).$ \end{proposition}

From this, we are able to deduce the following:

\begin{corollary}For any k $\geq 2$, $E(n,k) = n+1$.\end{corollary}
\begin{proof}Let $\zeta = \exp(\frac{2\pi i}{k})$. We consider two cases; first suppose that $n$ is even. Then for $j=1, \ldots, \frac{n}{2}$, we define the following $n+1$ polynomials:  $a_j(x) = (1+(-1)^{n-j}x^{n-j})(1+(-1)^{j}\zeta x^j)$, $b_j(x) = (1+(-1)^{n-j}x^{n-j})(1+(-1)^{j}\zeta^2 x^j)$ together with $1-\zeta x^n$, giving $n+1$ polynomials which have the form $g_{\sigma}(x)$ for some $\sigma \in C_k \wr S_n$. We claim these polynomials are linearly independent.

Suppose that for some constants $\alpha_j, \beta_j, \gamma$, we have $$\sum_{j=1}^{\frac{n}{2}} \alpha_j a_j(x) + \sum_{j=1}^{\frac{n}{2}} \beta_j b_j(x) + \gamma(1-\zeta x^n) = 0.$$ Equating coefficients of $x$ and $x^{n-1}$ gives us (after simplifying) the equations $\alpha_1\zeta + \beta_1\zeta^2 = 0$ and $\alpha_1 + \beta_1 = 0$, from which we deduce that $\alpha_1 = \beta_1 = 0$. We similarly deduce $\alpha_j = \beta_j = 0$ for $1 \leq j < \frac{n}{2}$ by equating coefficients of $x^j$ and $x^{n-j}$. This leaves three equations in $\alpha_{\frac{n}{2}}, \beta_{\frac{n}{2}}, \gamma$, from which it is trivial to show that these constants are all zero, thus proving the claim.

If, however, $n$ is odd, then we consider the polynomials $a_j(x)$, $b_j(x)$ for $j=0, 1, \ldots, \lfloor \frac{n}{2} \rfloor$, together with $1-\zeta x^n$ and $1- \zeta^2 x^n$. Again this gives $n+1$ polynomials which are of the form $g_{\sigma}(x)$, and a similar argument to that given for the previous case establishes that they are linearly independent.
\end{proof}

\section{Symmetric groups in prime characteristic}
\subsection{Symmetric group}
The arguments used in [1] worked over the complex numbers, and so it is natural to ask what happens in prime characteristic. To obtain a good theory of characters in this case, we need to work with Brauer characters. Let $\beta_{r,p}$ denote the Brauer character of the $r^{\text{th}}$ symmetric power of the natural representation of $S_n$ in characteristic $p$. 

For a full account of Brauer characters see [2, Chapter 15]; all we need to know about Brauer characters for our purposes is the following:

\begin{lemma} If $g \in S_n$ has order coprime to $p$, then $\beta_{r,p} (g) = \chi_r (g)$. \end{lemma}
\begin{proof}
This is a special case of the fact that the Brauer character of the $p$-modular reduction of a characteristic zero representation affording $\chi$ equals $\chi$ on the conjugacy classes of $p$-regular elements. For a proof, see [2, Theorem 15.8].
\end{proof}

Let $B(p,n)$ denote the dimension of the space spanned by the Brauer characters of the symmetric powers of the natural representation of $S_n$ in characteristic $p$. An element $g \in S_n$ has order coprime to $p$ if and only if its cycle type has all its parts coprime to $p$. Let $X_p$ denote the set of partitions of $n$ whose parts are all coprime to $p$. The usual argument shows that $\dim (\text{Span$_{\C}$}\lbrace f_\lambda : \lambda \in X_p \rbrace) = B(p,n)$.

We aim to put the rational functions $f_{\lambda}(x)$ (for $\lambda \in X_p$) over their least common denominator, which we denote by $g_p(x)$. If $g_p(x)$ has degree $\delta_p$, then the elements $\frac{1}{g_p(x)},\frac{x}{g_p(x)},\ldots,\frac{x^{\delta_p-n}}{g_p(x)}$ are a spanning set for $\text{Span$_{\C}$}\lbrace f_\lambda : \lambda \in X_p \rbrace$, giving $B(p,n) \leq \delta_p + 1 -n$. Our strategy is as follows: to work out what $g_p(x)$ is (up to a scalar), and then to bound its degree. 

\begin{definition}Let $k \in \N$, and write $k = p^am$ where $m$ is coprime to $p$. Then we define $r_p(k) = m$.\end{definition} 

\begin{proposition}We have $g_p(x) = \pm \prod_{k \leq n} (1-x^{r_p(k)})$. \end{proposition}
\begin{proof}For $\lambda \in X_p$, the denominator of $f_{\lambda}(x)$ is a polynomial of degree $n$ which factorises over $\Z$ as a product of cyclotomic polynomials. If $p$ divides $d$, then $\Phi_d(x)$ does not appear in this factorisation, whereas if $p$ does not divide $d$, then $\Phi_d(x)$ appears at most $\lfloor \frac{n}{k} \rfloor$ times in the common denominator. It follows that $$g_p(x) = \prod_{\substack{
      1 \le k \le n, \\
      (k,p)=1}} \Phi_k(x)^{\lfloor \frac{n}{k} \rfloor}.$$

Note that $\prod_{k \leq n} (1-x^{r_p(k)})$ is also a product of cyclotomic polynomials and $\Phi_d(x)$ appears if and only if $r_p(k)$ is a multiple of $d$. If $p$ divides $d$, then $r_p(k)$ is not a multiple of $d$; conversely, if $p$ does not divide $d$, then whenever $k$ is a multiple of $d$, $r_p(k)$ is a multiple of $d$. Since there are $\lfloor \frac{n}{d}\rfloor$ multiples of $d$ between $1$ and $n$, it follows that the $d^{\text{th}}$ cyclotomic polynomial appears $\lfloor \frac{n}{d}\rfloor$ times in $\prod_{k \leq n} (1-x^{r_p(k)})$. 
\end{proof}

\begin{proposition}The degree of $g_p(x)$ is at most $\frac{p}{2p+2}n^2 + n(\log_p n + 1)$. \end{proposition}
\begin{proof}
Let $m \in \N$ with $m \leq n$. The equation $r_p(k) = m$ has no solutions if $m$ is divisible by $p$ and $1 + \lfloor \log_p(\frac{n}{m}) \rfloor$ solutions if $p$ does not divide $m$, giving $$g_p(x) = \prod_{\substack{
      1 \le k \le n, \\
      (k,p)=1}} (1-x^k)^{1+\lfloor \log_p(\frac{n}{k}) \rfloor}.$$ 
Hence the degree of $g_p(x)$, namely $\delta_p$, is $$\sum_{\substack{
      1 \le k \le n, \\
      (k,p)=1}} k(1+\lfloor \log_p(\frac{n}{k}) \rfloor).$$
We can rewrite this as $$\delta_p = \sum_{\substack{
      1 \le k \le n, \\
      (k,p)=1}} k + \sum_{\substack{
      1 \le k \le \frac{n}{p}, \\
      (k,p)=1}} k + \sum_{\substack{
      1 \le k \le \frac{n}{p^2}, \\
      (k,p)=1}} k +\cdots .$$
It is easy to show that $$\sum_{\substack{
      1 \le m \le x, \\
      (m,p)=1}} m  \leq \frac{p-1}{2p}x^2 + x.$$
Therefore, we have $$\delta_p = \frac{p-1}{2p}(n^2 + \frac{n^2}{p^2} + \frac{n^2}{p^4} + \cdots) + n(\log_pn+1).$$  
We bound the sum in the brackets by the sum of the infinite geometric series, giving $\delta_p \leq \frac{p-1}{2p}\frac{n^2}{1-p^{-2}} + n(\log_pn+1).$
After some simple algebra, we see that $\delta_p \leq \frac{p}{2p+2}n^2 + n(\log_pn+1)$, as required.
\end{proof}

Recall that just before Definition 9, we established that $B(p,n) \leq \delta_p + 1 - n$, and so the following result is immediate:

\begin{corollary}For every prime number $p$ and $n \in \N$, we have $B(p,n) \leq \frac{p}{2p+2}n^2 + n\log_p n +1$. \end{corollary}

Hence we have obtained an upper bound on $B(p,n)$ which is asymptotic to $\frac{pn^2}{2p+2}$; this is always smaller than the $\frac{n^2}{2}$ obtained in the characteristic zero case in [1]. It is also worth noting that as $p$ becomes very large, this upper bound becomes close to the one obtained for characteristic zero. We now aim to find a lower bound on $B(p,n)$.

\begin{remark}It is, at first sight, natural to try to imitate the argument for a lower bound given in [1, Section 3] by considering the dimension of the vector space $\text{Span$_{\C}$} \lbrace e_\lambda(1,q,q^2,...) : \lambda \in X_p \rbrace$. However, it is not true that $\text{Span$_{\C}$} \lbrace u_\lambda(1,q,q^2,...) : \lambda \in X_p \rbrace = \text{Span$_{\C}$} \lbrace e_\lambda(1,q,q^2,...) : \lambda \in X_p \rbrace$ (if we instead allow $\lambda$ to be any partition, then it is true, and this is essential to the argument in [1]). While it is possible to modify the argument in [1] to prove that $\text{dim(Span$_{\C}$} \lbrace e_\lambda(1,q,q^2,...) : \lambda \in X_p \rbrace) \geq \frac{n^2}{2} - Kn^{\frac{3}{2}}$ for some constant $K$, which is of some independent interest, this does not help here.
Therefore, our strategy is to exhibit a large linearly independent subset of $\text{Span$_{\C}$}\lbrace f_\lambda : \lambda \in X_p \rbrace$.\end{remark}  

Given $n \in \N$ and a prime number $p$, we define $A_{p,n}$ to be the set of partitions of $n$ which have the form $(r^a, 1^b)$, where $r$ is coprime to $p$, and $a,b \in \N_0$.

\begin{proposition}We have $B(p,n) \geq \frac{p-1}{p} n\log n + n(\gamma - 1 - \frac{p-1}{p} - \frac{\gamma}{p})$, where $\gamma$ denotes the Euler--Mascheroni constant. \end{proposition}
\begin{proof}
Suppose that, in $\Q(x)$, we have $f_\lambda(x) = \sum_{\mu \in M} \alpha_\mu f_{\mu}(x)$, where $M$ is a set of partitions such that the denominator of $f_\lambda(x)$ contains the cyclotomic factor $\Phi_a(x)^b$, but $\Phi_a(x)^b$ does not divide the denominator of any $f_{\mu}(x)$ for $\mu \in M$. Putting the right-hand side over a common denominator and multiplying both sides by $\frac{1}{f_{\lambda}(x)\Phi_a(x)^b}$, we obtain the relation $ \frac{1}{\Phi_a(x)^b} = \frac{F(x)}{G(x)},$
where $\Phi_a(x)^b$ does not divide $G(x)$. But this gives, in $\Q[x]$, the relation $F(x)\Phi_a(x)^b = G(x)$ which contradicts uniqueness of factorisation in $\Q[x]$. Hence the type of linear dependence mentioned above is impossible. It follows that the set $\lbrace f_\lambda(x) : \lambda \in A_{p,n} \rbrace$ is linearly independent since, by construction, the denominator of each $f_\lambda(x)$ for $\lambda \in A_{p,n}$ contains a power of a cyclotomic polynomial which does not divide any other denominator. Thus $B(p,n) \geq |A_{p,n}|$.

It is easy to see that $$|A_{p,n}| = 1 + \sum_{\substack{
      1 < r \le n, \\
      (p,r)=1}} \lfloor \frac{n}{r} \rfloor.$$ Therefore, we have $$|A_{p,n}| \geq 1 + \sum_{\substack{
      1 < d \le n, \\
      (p,d)=1}} (\frac{n}{d} - 1) = 1 + n(H_n - 1 - \frac{1}{p}H_{\lfloor \frac{n}{p} \rfloor} ).$$ 
It follows from [3, Theorem 1] that $$|A_{p,n}| \geq 1 + \frac{p-1}{p} n\log n + n(\gamma - 1 - \frac{\gamma}{p}).$$ This gives the required result.
\end{proof}

\subsection{Wreath Products}
We conclude by unifying our two extensions and considering the case of the Brauer characters of the generalised symmetric group. Let $B(p,n,k)$ denote the dimension of the space spanned by the Brauer characters of the symmetric powers of the natural representation of $C_k \wr S_n$ in characteristic $p$ and let $\zeta =\exp(\frac{2\pi i}{k})$. An element $\sigma \in C_k \wr S_n$ has order coprime to $p$ if and only if both $\lambda(\sigma) \in X_p$ and every $t_i$ has order coprime to $p$. Let $R_{p,n,k}$ be the subset of $C_k \wr S_n$ consisting of elements of order coprime to $p$.  The usual argument shows that $B(p,n,k) = \dim (\Span_\C \lbrace f_\sigma(x) : \sigma \in R_{p,n,k} \rbrace) $. 

\begin{proposition}We have $B(p,n,k) \leq r_p(k)(\frac{p}{2p+2}n^2 + n(\log_p n+1)) - n +1$. \end{proposition}\begin{proof}Recall from the proof of Proposition 3 that a common denominator for the $f_\sigma(x)$ for $\sigma \in C_k \wr S_n$ is $$ D(x) = \prod_{\substack{
      0 \le i < k \\
      1 \le j \le n}} (1-\zeta^ix^j).$$ 
We say that we can cancel $p(x) \in \C[x]$ if $\frac{D(x)}{p(x)}$ is a common denominator for the $f_\sigma(x)$, where $\sigma \in R_{p,n,k}$. Our strategy is to cancel factors from $D(x)$. First, note that if $p$ divides the order of $\zeta^i$, then we may cancel $\prod_{j=1}^n (1-\zeta^ix^j)$ from $D(x)$ by definition of $R_{p,n,k}$. Fix $i$ such that $\zeta^i$ has order coprime to $p$. We need the following lemma.      
      
\begin{lemma}If $\sigma \in R_{p,n,k}$ and $x-\theta$ divides $\frac{1}{f_\sigma(x)}$ in $\C[x]$, then $\theta$ has order coprime to $p$. \end{lemma} 
\begin{proof}Suppose that  $x-\theta$ divides $\frac{1}{f_\sigma(x)}$ in $\C[x]$, then $\frac{1}{f_\sigma(\theta)} = 0$ and so $\prod_i(1-t_i\theta^{\lambda_i}) = 0$. Hence there exist $i,j$ such that $1-\zeta^i \theta^{\lambda_j} = 0$. Since $\sigma \in R_{p,n,k}$, $\zeta^i$ has order coprime to $p$ and $\lambda_j$ is coprime to $p$. But the order of $\theta$ divides $\text{ord}(\zeta^i)\lambda_j$, and so is coprime to $p$.
\end{proof}   
Therefore, in a factorisation of $\prod_{j=1}^n (1-\zeta^ix^j)$ into linear factors, we may cancel any factor $x-\theta$ where the order of $\theta$ is divisible by $p$. Consider the factor $1-\zeta^ix^j$, and write $j=p^er$ where $p$ does not divide $r$. The equation $1-\zeta^ix^j =0$ has $j$ roots, and exactly $r$ have order not divisible by $p$. It follows that we may cancel $j-r$ linear factors from the factorisation of $1-\zeta^ix^j$, leaving a term of degree $r = r_p(j)$, as defined in Definition 9.  

Consequently, cancelling factors from $\prod_{j=1}^n (1-\zeta^ix^j)$, replaces $1-\zeta^ix^j$ with a term of degree $r_p(j)$. This new polynomial has degree $\sum_{j=1}^n r_p(j)$, which is at most $\frac{p}{2p+2}n^2 + n(\log_p n + 1)$ by the proof of Proposition 11. We thus can put the power series $\lbrace f_\sigma(x) : \sigma \in R_{p,n,k} \rbrace$ over a common denominator of degree at most $r_p(k)(\frac{p}{2p+2}n^2 + n(\log_p n+1))$, giving the desired result.\end{proof} 
\subsection*{Acknowledgement} I am grateful to my advisor, Dr. Mark Wildon, for his many helpful comments and suggestions.
\section*{References}
\noindent [1] \textsc{David Savitt and Richard P. Stanley}, `A note on the symmetric powers of the standard representation of $S_n$', \textit{The Electronic Journal of Combinatorics} \textbf{7} (2000), R6

\noindent [2] \textsc{I. Martin Isaacs}, \textit{`Character Theory of Finite Groups'}, Dover Publications (1994)

\noindent [3] \textsc{Bai-Ni Guo and Feng Qi}, `Sharp bounds for harmonic numbers',
\textit{Applied Mathematics and Computation} 218 (2011), no. 3, 991–--995; Available online at http:
//dx.doi.org/10.1016/j.amc.2011.01.089

\noindent [4] \textsc{Gordon James and Adalbert Kerber}, \textit{`The Representation Theory of the Symmetric Group'}, Cambridge University Press (1984)

\end{document}